%% file: Upward_closure.tex
\documentclass[12pt]{amsart}

\title[Upward closure and amalgamation in the generic multiverse]{Upward closure and amalgamation in the generic multiverse of a countable model of set theory}

\author[Joel David Hamkins]{Joel David Hamkins}
 \address[J. D. Hamkins]
        {Mathematics, Philosophy, Computer Science,
          The Graduate Center of The City University of New York,
          365 Fifth Avenue, New York, NY 10016
          \&
          Mathematics,
          College of Staten Island of CUNY,
          Staten Island, NY 10314}
\email{jhamkins@gc.cuny.edu}
\urladdr{http://jdh.hamkins.org}

\thanks{This article is based upon I talk I gave at the conference on Recent Developments in Axiomatic Set Theory at the Research Institute for Mathematical Sciences (RIMS) at Kyoto University, Japan in September, 2015, and I am extremely grateful to my Japanese hosts, especially Toshimichi Usuba, for supporting my research visit there and also at the CTFM conference at Tokyo Institute of Technology just preceding it. My research has also been supported in part by the Simons Foundation (grant 209252), and part of this article was written while I was Visiting Fellow at the Isaac Newton Institute in Cambridge, UK for the program on the Mathematical, Computational and Foundational Aspects of the Higher Infinite. This article includes material adapted from section \S2 of~\cite{FuchsHamkinsReitz2015:Set-theoreticGeology}, jointly with G.~Fuchs, myself and J.~Reitz. Theorem~\ref{Theorem.ChainsOfCohenExtensionsHaveUpperBound} was proved in a series of conversations I had with Giorgio Venturi at the Young Set Theory Workshop 2011 in Bonn and continuing at the London 2011 summer school on set theory at Birkbeck University London. Related exposition and commentary can be found on my blog at~\cite{Hamkins2015:BlogPostUpwardClosureInTheToyMultiverseOfAllCountableModelsOfSetTheory, Hamkins2015:BlogPostUpwardClosureInTheGenericMultiverseOfForcingToAddACohenReal}. Commentary concerning this paper can be made at \url{http://jdh.hamkins.org/upward-closure-and-amalgamation-in-the-generic-multiverse}.}

\usepackage{amssymb}
\usepackage[hidelinks]{hyperref}

\include{MathMacrosJDH}

\begin{document}

\begin{abstract}
 I prove several theorems concerning upward closure and amalgamation in the generic multiverse of a countable transitive model of set theory. Every such model $W$ has forcing extensions $W[c]$ and $W[d]$ by adding a Cohen real, which cannot be amalgamated in any further extension, but some nontrivial forcing notions have all their extensions amalgamable. An increasing chain $W[G_0]\of W[G_1]\of\cdots$ has an upper bound $W[H]$ if and only if the forcing had uniformly bounded essential size in $W$. Every chain $W\of W[c_0]\of W[c_1]\of \cdots$ of extensions adding Cohen reals is bounded above by $W[d]$ for some $W$-generic Cohen real $d$.
\end{abstract}
\maketitle

\noindent
Consider a countable transitive model of set theory $W\satisfies\ZFC$ in the context of all its forcing extensions. Several natural questions immediately suggest themselves concerning issues of amalgamation and upward-closure. For example, can any two such models be amalgamated into a common larger model? In other words, is this collection of models upward directed? When can we expect to find upper bounds for increasing chains? In this article, I shall resolve these and other similar questions. In particular, theorem~\ref{Theorem.Nonamalgamation} shows that there are forcing extensions $W[c]$ and $W[d]$, each adding a Cohen real, which have no common further extension; theorem~\ref{Theorem.WideForcingNonamalgamation} generalizes this non-amalgamation phenomenon to a wide class of other forcing notions, but theorem~\ref{Theorem.AutomaticMutualGenericity} shows that some forcing notions do always admit amalgamation. For upward closure, theorem~\ref{Theorem.ChainsOfCohenExtensionsHaveUpperBound} shows that every chain $$W[c_0]\of W[c_1]\of W[c_2]\of \cdots$$ of Cohen-real extensions of $W$ has an upper bound $W[d]$ in another Cohen-real extension, and theorem~\ref{Theorem.ChainUpperBoundIff} shows generally that any chain of forcing extensions has an upper bound if and only if the forcing was uniformly bounded in essential size.

In order to make a self-contained presentation, this article includes several results adapted from my previous joint work with Gunter Fuchs and Jonas Reitz~\cite[\S2]{FuchsHamkinsReitz2015:Set-theoreticGeology}, as well as some joint work with Giorgio Venturi.

\section{The generic multiverse}

Before presenting the results, let me briefly place the work into a somewhat broader context, which furthermore has connections with issues in the philosophy of set theory. Namely, the forcing extensions of a fixed model of set theory $W$ form an upward oriented cone in what is called the {\df generic multiverse} of $W$, which is the collection of all models that one can reach from $W$ by iteratively moving either to a forcing extension or a ground model, in each case by set forcing in the relevant model. Thus, every model $M$ in the generic multiverse of $W$ is reachable by a zig-zag path of models, where at each step we take either a forcing extension or a ground. The generic multiverse of $W$ itself can be viewed as a small part, a local neighborhood, of any of the much larger collections of models that express fuller multiverse conceptions. For example, one could look at the class-forcing multiverse, arising by closing $W$ under class forcing extensions and grounds, or the pseudo-ground multiverse, obtained by closing under pseudo-grounds, or the multiverse arising by closing under arbitrary extensions and inner models, and so on.

These collections of models, each a toy multiverse, if you will, offer various mathematically precise contexts in which one may investigate multiverse issues. Questions that begin philosophically, perhaps concerning the nature of what one might imagine as the full actual multiverse---the multiverse in which our (current) set-theoretic universe is one amongst many other set-theoretic worlds instantiating all the various alternative concepts of set that we might adopt---are transformed into analogous but mathematically precise questions in the toy multiverses, and we may hope to settle them. In this way, philosophical contemplation becomes mathematical investigation, and each toy multiverse serves as a proxy for the full actual multiverse.\footnote{I discuss this proxy idea further in the final parts of~\cite{Hamkins2014:MultiverseOnVeqL}, but see also~\cite{Hamkins2012:TheSet-TheoreticalMultiverse, Hamkins2011:TheMultiverse:ANaturalContext, Hamkins2009:SomeSecondOrderSetTheory}.} This article is an instance of the process: by presenting the mathematical solutions to several natural questions about closure and amalgamation in the case of the generic multiverse of a given countable transitive model of set theory $W$, we hope to gain insight about what might be true in the multiverse of $V$.

One may view the generic multiverse of $W$ as a Kripke model of possible worlds, connected by the forcing extension and ground model relations as a notion of accessibility, and this perspective leads one to consider the modal logic of forcing (see~\cite{HamkinsLoewe2008:TheModalLogicOfForcing,HamkinsLoewe2013:MovingUpAndDownInTheGenericMultiverse, Hamkins2003:MaximalityPrinciple}). An open question arising from that work is the following:

\begin{question}\label{Question.Inclusion=ForcingExtension?}
 Does the inclusion relation coincide with the ground-model/forcing-extension relation in the generic multiverse? That is, if $M$ is in the generic multiverse of $W$ and $W\of M$, must $M$ be a forcing extension of $W$?
\end{question}

A related open question concerns downward directedness:

\begin{question}\label{Question.DDG?}
 If $M$ and $N$ have a common forcing extension, must they have a common ground model?
\end{question}

In other words, if $M[G]=N[H]$ for $M$-generic $G\of\P\in M$ and $N$-generic $H\of\Q\in N$, then must there be model $W$ such that $M=W[g]$ and $N=W[h]$ both arise as forcing extensions of $W$? The {\df downward directed grounds hypothesis} (\DDG) is the axiom asserting that any two ground models of the universe have a common deeper ground. Although it may appear to involve a second-order quantifier, over grounds, in fact this axiom is first-order expressible in the language of set theory, using the uniform definition of the ground models (see~\cite{FuchsHamkinsReitz2015:Set-theoreticGeology}). Indeed, there is an indexed parameterizaton $W_r$ for all sets $r$ of all the ground models of $V$ by set forcing, and so one may also formulation the set-directed strengthening of the \DDG, which asserts that for any set $I$, there is a ground model $W_s$ contained in every $W_r$ for $r\in I$.

The two questions are connected by the following fact.

\begin{theorem}\label{Theorem.DDGImpliesInclusion=ForcingExtension}
 If the downward directed grounds hypothesis holds throughout the generic multiverse of $W$, then inclusion coincides with the ground-model/forcing-extension relation in that generic multiverse.
\end{theorem}

\begin{proof}
The \DDG\ assumption implies that whenever one has a ground model of a forcing extension, then it is also a forcing extension of a ground model. Thus, the \DDG\ in the generic multiverse of $W$ implies that the zig-zag paths need never go up and then down, that is, from a model $M$ up to a forcing extension $M[G]$ and then down to a ground model $N\of M[G]$, because since $N$ is a ground of $M[G]$ there is some $N$-generic filter $H\of\Q\in N$ for which $N[H]=M[G]$, and so by the \DDG\ there is a common ground $U\of M\cap N$ such that $M=U[g]$ and $N=U[h]$. So one could have gotten from $M$ to $N$ by going down to $U$, and then up to $U[h]=N$. Thus, the generic multiverse of $W$ consists of the forcing extensions $U[g]$ of the grounds $U$ of $W$. And if one such model $U[g]$ is contained in another $U[h]$, then $U\of U[g]\of U[h]$, so that $U[g]$ is an intermediate \ZFC\ model between a ground model $U$ and a forcing extension $U[h]$. It now follows by the intermediate model theorem (see~\cite[cor.~15.43]{Jech:SetTheory3rdEdition}, also~\cite[fact~11]{FuchsHamkinsReitz2015:Set-theoreticGeology}) that $U[g]$ is a ground of $U[h]$ by a quotient of the forcing giving rise to $U\of U[h]$.
\end{proof}

Toshimichi Usuba has very recently announced a proof of the downward directed grounds \DDG\ hypothesis, and indeed, of the strong \DDG\ in \ZFC, which is very welcome and exciting news, and this will settle question~\ref{Question.DDG?} as well as question~\ref{Question.Inclusion=ForcingExtension?}, in light of theorem~\ref{Theorem.DDGImpliesInclusion=ForcingExtension}.

\section{Non-amalgamation in the generic multiverse}

Let's begin with the basic non-amalgamation result, which I first heard from W. Hugh Woodin in the 1990s.

\begin{theorem}[Woodin, {\cite[obs.~35]{FuchsHamkinsReitz2015:Set-theoreticGeology}}]\label{Theorem.Nonamalgamation}
 If $W$ is any countable transitive model of set theory, then there are $W$-generic Cohen reals $c$ and $d$, for which the corresponding forcing extensions $W[c]$ and $W[d]$ have no common extension to a model of set theory with the same ordinals.
\end{theorem}

\begin{proof}
Let us view Cohen forcing as the partial order $2^\ltomega$ consisting of finite binary sequences ordered by extension. Enumerate the dense subsets of this forcing in $W$ as $\<D_n\mid n<\omega>$. Fix a real $z\in 2^\omega$ that could not possibly exist in any forcing extension of $W$, such as a real coding a relation on $\omega$ with order type $\Ord^W$. We shall now build the reals $c,d\in 2^\omega$ in stages, with $$c=\Union_n c_n\qquad d=\Union_n d_n,$$ where $c_n$ and $d_n$ are the finite binary initial segments of $c$ and $d$, respectively, that have been specified by stage $n$. We undertake the construction in such a way that $c_n$ and $d_n$ are each in $D_n$, so that the reals individually are $W$-generic, but from $c$ and $d$ together, we can compute $z$. To begin, let $c_0$ be any element of $D_0$, and let $d_0$ consist of $|c_0|$ many $0$s, followed by a $1$ and then the $0^{\rm th}$ bit $z(0)$ of $z$, and then extended so that $d_0\in D_0$. Next, we extend $c_0$ by padding with $0$s until it has the length of $d_0$, and then a $1$, and then the next bit $z(1)$, followed by an extension to $c_1$ that is in $D_1$. Now form $d_1$ by padding $d_0$ with $0$s until the length of $c_1$, followed by a $1$, followed by $z(2)$, and then extended to an element $d_1\in D_1$. And so on in this same pattern. Since the sequences $c_n$ and $d_n$ are in $D_n$, it follows that both $c$ and $d$ will be $W$-generic Cohen reals. But notice that if we have both $c$ and $d$ together, then because the padding with $0$s exactly identifies the coding points, we can therefore reconstruct the construction history $c_n$ and $d_n$ and therefore compute the real $z$. So there can be no common extension $W[c],W[d]\of U$ to a model of \ZFC\ with the same ordinals, as if both $c,d\in U$, then $z$ would also be in $U$, contrary to our assumption on $z$.
\end{proof}

The same argument generalizes to construct three $W$-generic Cohen reals $c,d,e$ such that any two of them are mutually $W$-generic, but the three models $W[c], W[d], W[e]$ have no common extension with the same ordinals. And more generally:

\begin{theorem}[{\cite[obs.~36]{FuchsHamkinsReitz2015:Set-theoreticGeology}}]\label{Theorem.Nonamalgamation-n-reals}
 If $W$ is any countable transitive model of set theory, then for any finite $n$ there are distinct $W$-generic Cohen reals $c_0,\ldots,c_n$, any proper subset of which is mutually $W$-generic, but the models $W[c_i]$ altogether have no common extension to a model of set theory with the same ordinals as $W$.
\end{theorem}

\begin{proof}
Build the reals $c_k=\Union_s c_{k,s}$ in stages. Fix a bad real $z$, which cannot exist in any extension of $W$ with the same ordinals. Enumerate the dense sets $D_s$ of $W$ for the forcing $\Add(\omega,n)$ to add $n$ many Cohen reals. At a given stage, consider each $i\leq n$ in turn and extend all the other $c_{j,s}$ for $j\neq i$ in such a way so as to ensure that $\<c_j>_{j\neq i}$ is in $D_s$, and then pad them all with $0$s to make them all have the same length; pad $c_{j,s}$ with $0$s also to this length, followed by a $1$, followed by the next digit of $z$. In this way, $\<c_j>_{j\neq i}$ is $W$-generic for adding $n$ many Cohen reals, so they are mutually $W$-generic, but the whole collection $\<c_j>_j$ computes the construction history and also the forbidden real $z$, and therefore cannot exist in any extension of $W$ with the same ordinals.
\end{proof}

Let us consider whether this pattern continues into the infinite.

\begin{question}\label{Question.InfiniteNonamalgamation?}
 If $W$ is a countable transitive model of set theory, must there be $W$-generic Cohen reals $\<c_n\mid n<\omega>$, such that any finitely many of them are mutually $W$-generic, but the models $W[c_n]$ for all $n<\omega$ have no common extension to a model of set theory with the same ordinals?
\end{question}

The answer, provided by theorem~\ref{Theorem.ChainUpperBoundIff} and more forcefully by theorem~\ref{Theorem.ChainsOfCohenExtensionsHaveUpperBound}, is that no, in this infinite case we have amalgamation: every increasing chain of Cohen-real extensions $W[c_n]$ is bounded above by $W[d]$ for some $W$-generic Cohen real $d$, so that $W[c_n]\of W[d]$ for all $n$.

\begin{question}\label{Question.GeneralNonamalgamation?}
 Does the nonamalgamation result of theorem~\ref{Theorem.Nonamalgamation} hold for other forcing notions? Does every nontrivial forcing notion exhibit non-amalgamation?
\end{question}

In other words, if $W$ is a countable transitive model of set theory and $\Q\in W$ is a nontrivial notion of forcing, are there $W$-generic filters $g,h\of\Q$ such that $W[g]$ and $W[h]$ have no common forcing extension?

The first thing to say about question~\ref{Question.GeneralNonamalgamation?} is that there is a large class of forcing notions $\Q$ for which the non-amalgamation phenomenon occurs. In particular, the reader may observe as an exercise that the proof of theorem~\ref{Theorem.Nonamalgamation} directly generalizes to many other forcing notions, such as adding Cohen subsets to higher cardinals, or collapsing cardinals to $\omega$ or to another cardinal. Let us push this a bit further, however, by defining that a notion of forcing $\Q$ is {\df wide}, if it is not $|\Q|$-c.c. below any condition. In other words, $\Q$ is wide, if below every condition $q\in\Q$, there is an antichain in $\Q\restrict q$ of the same size as $\Q$. Many commonly considered forcing notions are wide, and these all exhibit the non-amalgamation phenomenon.

\begin{theorem}\label{Theorem.WideForcingNonamalgamation}
 If $W$ is a countable transitive model of \ZFC\ and $\Q$ is a nontrivial notion of forcing that is wide in $W$, then:
 \begin{enumerate}
  \item There are $W$-generic filters $g,h\of\Q$, such that the corresponding forcing extensions $W[g]$ and $W[h]$ have no common extension to a model of set theory with the same ordinals as $W$.
  \item Indeed, for any finite number $n$, there are $W$-generic filters $g_0,\ldots,g_n\of\Q$, such that any proper subset of them are mutually $W$-generic, but there is no common extension of all the $W[g_k]$ to a model of set theory with the same ordinals as $W$.
  \item Furthermore, it suffices for these conclusions that $\Q$ should have merely a nontrivial subforcing notion that is wide.
 \end{enumerate}
\end{theorem}

\begin{proof} Consider first just the first case, where we have two generic filters. Enumerate the dense subsets of $\Q$ in $W$ as $\<D_n\mid n<\omega>$, and using the wideness of $\Q$, fix in $W$ an assignment to each condition $q\in\Q$ a maximal antichain $A_q\of\Q\restrict q$ and an enumeration of it as $\<q^{(\alpha)}\mid\alpha<|\Q|>$. Fix also an enumeration of $\Q$ in order type $|\Q|$, which we may assume is an infinite cardinal in $W$. Outside of $W$, fix a bad real $z$, which cannot exist in any extension of $W$ to a model of set theory with the same ordinals, such as a real coding the ordinals of $W$. We shall construct $g$ and $h$ to be the respective filters generated by the descending sequences $p_0\geq p_1\geq\cdots$ and $q_0\geq q_1\geq\cdots$, choosing $p_n,q_n\in D_n$. Begin with any $p_0,q_0\in D_0$. If $p_n$ and $q_n$ are defined, then let $\alpha<|\Q|$ be the ordinal for which $q_n$ is the $\alpha^{\rm th}$ element of $\Q$. We first extend $p_n$ to the $(2\cdot\alpha+z(n))^{\rm th}$ element of $A_{p_n}$, thereby coding $\alpha$ and the value of $z(n)$, and then extend further to a condition $p_{n+1}\in D_{n+1}$. Next, on the other side, we extend $q_n$ by picking the $\beta^{\rm th}$ element of $A_{q_n}$, where $p_{n+1}$ is the $\beta^{\rm th}$ element of $\Q$, and then extend further to $q_{n+1}\in D_{n+1}$. In this way, the filters $g$ and $h$ generated respectively by the $p_n$ and $q_n$ will each be $W$-generic, but in any extension of $W$ that has both $g$ and $h$, we will be able to recover the map $n\mapsto\<p_n,q_n,z(n)>$, because if we know $p_n$, then the way that $g$ meets $A_{p_n}$ determines both $q_n$ and $z(n)$, and the way that $h$ meets $A_{q_n}$ determines $p_{n+1}$. So any extension of $W$ with both $g$ and $h$ also has $z$, which by assumption cannot exist in any extension of $W$ with the same ordinals. So $W[g]$ and $W[h]$ are non-amalgamable, as desired.

Just as in theorem~\ref{Theorem.Nonamalgamation-n-reals}, the argument generalizes to the case of adding any finite number of $W$-generic filters $g_0,\ldots,g_n\of\Q$, such that if one omits any one of them, the result is $W$-generic for $\Q\times\cdots\times\Q$, but the full sequence cannot exist in any extension of $W$ with the same ordinals. One fixes a bad real $z$, and then enumerates the dense sets for the $n$-fold product $\Q^n$, extending all but one so as to meet the relevant dense set, extending the excluded condition into its antichain so as to code the information that was just added by extending the other conditions, and then also coding one more bit of $z$. Omitting any one filter will result in a $W$-generic product of filters, but if one has all of them, then one can reconstruct the entire construction history and therefore also $z$.

Finally, let us suppose merely that $\Q$ has a subforcing notion $\Q_0\of\Q$ that is wide. By what we have proved already, we may find $g_0,h_0\of\Q_0$ which are $W$-generic for $\Q_0$, but are non-amalgamable over $W$. Next, we may find $W[g_0]$-generic and $W[h_0]$-generic filters $g/g_0$ and $h/h_0$, respectively, for the quotient forcing. It follows that $W[g]$ and $W[h]$ are non-amalgamable, since any extension of them would also amalgamate $W[g_0]$ and $W[h_0]$.
\end{proof}

The third claim of the theorem is relevant, for example, in the case of the \Levy\ collapse of an inaccessible cardinal $\kappa$. This forcing is not wide, because it has size $\kappa$ and is $\kappa$-c.c., but the \Levy\ collapse does have numerous small wide forcing factors---for example, it adds a Cohen real---and these are sufficient to cause the non-amalgamation phenomenon.

Meanwhile, the answer to the second part of question~\ref{Question.GeneralNonamalgamation?} is negative, because some forcing notions can always amalgamate their generic filters. Specifically, let us define that a forcing notion $\Q$ exhibits {\df automatic mutual genericity over $W$}, if whenever $g,h\of\Q$ are distinct $W$-generic filters, then they are mutually generic, so that $g\times h$ is $W$-generic for $\Q\times\Q$. In this case, both $W[g]$ and $W[h]$ would be contained in $W[g\times h]$, which would be a forcing extension of $W$ amalgamating them. Internalizing the concept to \ZFC, let us define officially that a forcing notion $\Q$ exhibits automatic mutual genericity, if in every forcing extension of $V$, any two distinct $V$-generic filters $G,H\of\Q$ are mutually $V$-generic for $\Q$. (This is first-order expressible in the language of set theory.) It is easy to see that if $\Q$ has the property that whenever $p\perp q$ and $p$ forces that $\dot D$ is dense below $\check q$, then there is a set $D$ in the ground model that is dense below $q$, and a strengthening $p'\leq p$ such that $p'$ forces $\check D\of\dot D$, then $\Q$ exhibits automatic mutual genericity over the ground model. This is a rigidity concept, since if $\Q$ has nontrivial automorphisms, or even if two distinct cones in $\Q$ are forcing equivalent, then clearly it cannot exhibit automatic mutual genericity, since mutually generic filters are never isomorphic by a ground-model isomorphism.

\begin{theorem}\label{Theorem.AutomaticMutualGenericity}
 If $\Diamond$ holds, then there is a notion of forcing that exhibits automatic mutual genericity.
 If there is a transitive model of \ZFC, then there is one $W$ with a notion of forcing $\Q$, such that any two distinct $W$-generic filters $g,h\of\Q$ are mutually generic and hence amalgamable by $W[g\times h]$.
\end{theorem}

\begin{proof}
By~\cite[thm.~2.6]{FuchsHamkins2009:DegreesOfRigidity}, it follows that $\Diamond$ implies that there is a Suslin tree $T$ on $\omega_1$ that is {\df Suslin off the generic branch}, in the sense of~\cite[def.~2.2]{FuchsHamkins2009:DegreesOfRigidity}, which means that after forcing with $T$, which adds a generic branch $b\of[T]$, the tree remains Suslin below any node that is not on $b$. (A generic Suslin tree also has this property; see~\cite[thm.~2.3]{FuchsHamkins2009:DegreesOfRigidity}.) If a tree is Suslin off the generic branch, then it must also have the unique branch property---forcing with it adds exactly one branch---since a second branch would contradict the Suslinity of that part of the tree, and thus, this property is a strong form of rigidity. But more, such a tree used as a forcing notion exhibits automatic mutual genericity. To see this, suppose that $g,h\of T$ are distinct $V$-generic filters for this forcing, individually. Let $p\in T$ be a node of the tree that lies on $h$, but not $g$. Since the tree was Suslin off the generic branch in $V$, it follows that $T_p$, the part of $T$ consisting of nodes comparable with $p$, is a Suslin tree in $V[g]$. Thus, every antichain of $T_p$ in $V[g]$ is refined by a level of the tree. Since $h$ is a cofinal branch through $T_p$, it follows that $h$ meets every level of the tree and hence also every antichain in $V[g]$. So $h$ is $V[g]$-generic and thus they are mutually generic.

For the second claim, if there is a countable transitive model of \ZFC, then there is one $W$ satisfying $\Diamond$, which therefore has a tree that is Suslin off the generic branch. Thus, any two distinct $W$-generic filters $g,h\of\Q$ are mutually generic and so $W[g]$ and $W[h]$ are amalgamated by $W[g\times h]$, which is a forcing extension of $W$.
\end{proof}

The proof of theorem~\ref{Theorem.AutomaticMutualGenericity} shows that it is relatively consistent with \ZFC\ that there is a forcing notion exhibiting automatic mutual genericity and hence supporting amalgamation, but the argument doesn't settle the question of whether such kind of forcing exists in every model of set theory.

\begin{question}
 Is it consistent with \ZFC\ that there is no forcing notion with automatic mutual genericity? 
\end{question}

Note that there are other weaker kinds of necessary amalgamation. For example, if $c$ is a $W$-generic Cohen real and $A\of\omega_1^W$ is $W$-generic for the forcing to add a Cohen subset of $\omega_1$, then $c$ and $A$ are mutually generic, because the forcing to add $A$ is countably closed in $W$ and therefore does not add new antichains for the forcing $\Add(\omega,1)$ to add $c$. This phenomenon extends to many other pairs of forcing notions $\P$ and $\Q$, such that any $W$-generic filters $g\of\P$ and $h\of\Q$ are necessarily mutually generic.

\section{Upward closure in the generic multiverse}

Let us turn now to the question of upward closure. Suppose that we have a countable increasing chain of forcing extensions $$W\of W[G_0]\of W[G_1]\of W[G_2]\of\cdots,$$ where $W$ is a countable transitive model of set theory.

\begin{question}\label{Question.ChainUpperBound?}
Under which circumstances may we find an upper bound, a forcing extension $W[H]$ for which $W[G_n]\of W[H]$ for all $n<\omega$?
\end{question}

The question is answered by theorem~\ref{Theorem.ChainUpperBoundIff}, which provides a necessary and sufficient criterion. It is easy to see several circumstances where there can be no such upper bound. For example, if the extensions $W[G_n]$ collapse increasingly large initial segments of $W$, in such a way that every cardinal of $W$ is collapsed in some $W[G_n]$, then obviously we cannot find an extension of $W$ to a model of \ZFC\ with the same ordinals as $W$. It is also easy to see that in general, we cannot require that $\<G_n\mid n<\omega>\in W[H]$; this is simply too much to ask. For example, if every $G_n$ is a $W$-generic Cohen real, then we could flip the initial bits of each $G_n$ in such a way that the resulting infinite sequence $\<G_n\mid n<\omega>$ was coding an arbitrary real $z$, even though such a change would not affect the models $W[G_n]$, since each $G_n$ individually was changed only finitely. This issue is discussed at length in~\cite{FuchsHamkinsReitz2015:Set-theoreticGeology}.

Following ideas in~\cite{HamkinsLeibmanLoewe2015:StructuralConnectionsForcingClassAndItsModalLogic}, let us define that the {\df forcing degree} of a forcing extension $W\of W[H]$, where $H\of\P\in W$ is $W$-generic, is the smallest size in $W$ of the Boolean completion of a forcing notion $\Q\in W$ for which there is a $W$-generic filter $G\of\Q$ for which $W[G]=W[H]$. Thus, the forcing degree of a forcing extension is the smallest size of a complete Boolean algebra realizing that extension as a forcing extension.

\begin{theorem}\label{Theorem.ChainUpperBoundIff}
 Suppose that $W$ is a countable transitive model of \ZFC\ and that
  $$W\of W[G_0]\of W[G_1]\of W[G_2]\of\cdots\of W[G_n]\of\cdots$$
 is an increasing chain of forcing extensions $W[G_n]$, where $G_n\of\Q_n$ is $W$-generic. Then the following are equivalent:
 \begin{enumerate}
  \item The chain is bounded above by a forcing extension $W[H]$, for some forcing notion $\Q\in W$ and $W$-generic filter $H\of\Q$.
  \item The forcing degrees of the extensions $W\of W[G_n]$ are bounded in $W$
 \end{enumerate}
\end{theorem}

\begin{proof}
($2\to 1$). This direction is essentially~\cite[thm.~39]{FuchsHamkinsReitz2015:Set-theoreticGeology}, but I shall sketch the argument. Let us first handle the case of product forcing, rather than iterated forcing, the case for which we have a tower with the form $$W\of W[g_0]\of W[g_0][g_1]\of W[g_0][g_1][g_2]\of\cdots,$$ where the $g_n\of\P_n\in W$ are finitely mutually generic over $W$, and the $\P_n$ are uniformly bounded in size by a cardinal $\gamma$ in $W$. Let $\theta>\gamma$ be a sufficiently large regular cardinal in $W$ so that we may enumerate $\<\R_\alpha\mid\alpha<\theta>$ in $W$ all the possible forcing notions in $W$ of size at most $\gamma$, up to isomorphism, with unbounded repetition. Let $\R=\prod_\alpha\R_\alpha$ be the finite support product. This forcing has the $\gamma^+$-chain condition. Let $H\of\R$ be any $\Union_n W[g_0\times\cdots\times g_n]$-generic filter. Select a cofinal sequence $\<\theta_n\mid n<\omega>$ converging to $\theta$, for which $\R_{\theta_n}=\P_n$, and modify the filter $H$ to use $g_n$ at coordinate $\theta_n$ instead of what $H$ had there. If $H^*$ is the new filter, then $H(\theta_n)=g_n$, but at all other coordinates it agrees with $H$. Since $\R$ is $\gamma^+$-c.c., it follows that any maximal antichain for $\R$ in $W$ has bounded support, and thus interacts with only finitely many of the coordinates $\theta_n$ upon which we performed surgery. But $H$ is mutually generic with those finitely many $g_n$, and so that finite amount of surgery will preserve genericity. So $W[H^*]$ is a forcing extension of $W$, and every $g_n\in W[H^*]$ by construction. So $W[g_0\times\cdots\times g_n]\of W[H^*]$, as desired. For the general case, where we have iterated forcing rather than product forcing, consider a tower $W\of W[G_0]\of W[G_1]\of\cdots$, where each $G_n\of\Q_n\in W$ is $W$-generic and the $\Q_n$ are bounded in size. By collapsing the bound, and furthermore using a filter $g$ for the collapse that is not only $W$-generic, but also $W[G_n]$-generic for every $n$---this is possible because there are still only countably many dense sets altogether in $\Union_n W[G_n]$---we produce a larger tower $W\of W[g]\of W[g][G_0]\of W[g][G_1]\of\cdots$, where now the forcing $\Q_n$ is countable in $W[g]$ and thus isomorphic to the forcing to add a Cohen real there. By quotient forcing, we may therefore view this larger tower as $W\of W[g]\of W[g][c_0]\of W[g][c_0][c_1]\of W[g][c_0][c_1][c_2]\of\cdots$, where $W[g][G_n]=W[g][c_0\times\cdots\times c_n]$. Thus, we have reduced to the case of product forcing, for which we have already explained how to find an upper bound $W[H]$ containing every $W[g][c_o\times\cdots\times c_n]$ and hence also every $W[G_n]$ in the original tower.

($1\to 2$). This direction is similar to~\cite[lemma~23]{HamkinsLeibmanLoewe2015:StructuralConnectionsForcingClassAndItsModalLogic}, which was used in the context of the modal logic of forcing to show that the value of the forcing degree of a model over a fixed ground model is a {\df ratchet}, which is to say, that it can be made larger, but never smaller, with further forcing. Suppose that we have a tower $W\of W[G_0]\of W[G_1]\of\cdots$, which is bounded above by the forcing extension $W[H]$, where $H\of\Q\in W$ is $W$-generic. Since $W\of W[G_n]\of W[H]$, it follows by the intermediate model result of~\cite[lemma~15.43]{Jech:SetTheory3rdEdition} that there is a complete subalgebra $\C\of\B$, where $\B$ is the Boolean completion of $\Q$ in $W$, such that $W[G_n]=W[G']$ for some $W$-generic filter $G'\of\C$. Thus, the forcing degree of the extension $W\of W[G_n]$ is bounded by the size of $|\B|^W$, and this does not depend on $n$. So the extensions have uniformly bounded forcing degrees over $W$.
\end{proof}

In the result of theorem~\ref{Theorem.ChainUpperBoundIff}, the upper bound of $W[H]$ provided by the proof involves possibly collapsing a lot of cardinals, but we might not want to do that. For example, in question~\ref{Question.InfiniteNonamalgamation?} we have a tower of extensions $$W\of W[c_0]\of W[c_1]\of W[c_2]\of\cdots,$$ where each $c_n$ is a $W$-generic Cohen real, and we'd like to know whether we can find an upper bound also of this form. A close inspection of the proof of $(2\to 1)$ in theorem~\ref{Theorem.ChainUpperBoundIff} shows that we can dispense with the collapse forcing in this case, but the rest of the argument involves an uncountable product $\R$ of Cohen-real forcing; we can actually use $\Add(\omega,\omega_1)$ in that argument for this case. So the proof does not directly produce an upper bound in the form $W[d]$ of adding a single Cohen real.

Nevertheless, it is true that we can find an upper bound of this form, and this is what I shall now prove in theorem~\ref{Theorem.ChainsOfCohenExtensionsHaveUpperBound}. Specifically, I claim that the collection of models $M[c]$ obtained by adding an $M$-generic Cohen real $c$ over a fixed countable transitive model of set theory $M$ is upwardly countably closed, in the sense that every increasing countable chain has an upper bound. I proved this theorem with Giorgio Venturi back in 2011 in a series of conversations at the \href{http://jdh.hamkins.org/an-introduction-to-boolean-ultrapowers-bonn-2011/}{Young Set Theory Workshop} in Bonn and continuing at the \href{http://jdh.hamkins.org/a-tutorial-in-set-theoretic-geology/}{London summer school on set theory}.

\begin{theorem}\label{Theorem.ChainsOfCohenExtensionsHaveUpperBound}
 For any countable transitive model $W\satisfies\ZFC$, the collection of all forcing extensions $W[c]$ by adding a $W$-generic Cohen real is upward-countably closed. That is, for any countable tower of such forcing extensions
$$W\of W[c_0]\of W[c_1]\of\cdots\of W[c_n]\of\cdots,$$
we may find a $W$-generic Cohen real $d$ such that $W[c_n]\of W[d]$ for every natural number $n$.
\end{theorem}

\begin{proof}
Suppose that we have such a tower of forcing extensions $W[c_0]\of W[c_1]\of W[c_2]$, and so on. Note that if $W[b]\of W[c]$ for $W$-generic Cohen reals $b$ and $c$, then $W[c]$ is a forcing extension of $W[b]$ by a quotient of the Cohen-real forcing. But since the Cohen forcing itself has a countable dense set, it follows that all such quotients also have a countable dense set, and so $W[c]=W[b][b_1]$ for some $W[b]$-generic Cohen real $b_1$. Thus, we may view the tower as having the form:
$$W[b_0]\of W[b_0\times b_1]\of\cdots\of W[b_0\times b_1\times\cdots\times b_n]\of\cdots,$$
where now it follows that any finite collection of the reals $b_i$ are mutually $W$-generic.

Of course, we cannot expect in general that the real $\<b_n\mid n<\omega>$ is $W$-generic for $\Add(\omega,\omega)$, since this real may be very badly behaved. For example, the sequence of first-bits of the $b_n$'s may code a very naughty real $z$, which cannot be added by forcing over $W$ at all. So in general, we cannot allow that this sequence is added to the limit model $W[d]$. (See further discussion in my blog post~\cite{Hamkins2015:BlogPostUpwardClosureInTheToyMultiverseOfAllCountableModelsOfSetTheory}.) We shall instead undertake a construction by making finitely many changes to each real $b_n$, resulting in a real $d_n$, in such a way that the resulting combined real $d=\oplus_n d_n$ is $W$-generic for the forcing to add $\omega$-many Cohen reals, which is of course isomorphic to adding just one. To do this, let's get a little more clear with our notation. We regard each $b_n$ as an element of Cantor space $2^\omega$, that is, an infinite binary sequence, and the corresponding filter associated with this real is the collection of finite initial segments of $b_n$, which will be a $W$-generic filter through the partial order of finite binary sequences $2^{<\omega}$, which is one of the standard isomorphic copies of Cohen forcing. We will think of $d$ as a binary function on the plane $d:\omega\times\omega\to 2$, where the $n^{th}$ slice $d_n$ is the corresponding function $\omega\to 2$ obtained by fixing the first coordinate to be $n$.

Now, we enumerate the countably many open dense subsets of $W$ for the forcing to add a Cohen real $\omega\times\omega\to 2$ as $D_0$, $D_1$, and so on. Now, we construct $d$ in stages. Before stage $n$, we will have completely specified $d_k$ for $k<n$, and we also may be committed to a finite condition $p_{n-1}$ in the forcing to add $\omega$ many Cohen reals. We consider the dense set $D_n$. We may factor $\Add(\omega,\omega)$ as $\Add(\omega,n)\times\Add(\omega,[n,\omega))$. Since $d_0\times\cdots\times d_{n-1}$ is actually $W$-generic (since these are finite modifications of the corresponding $b_k$'s, which are mutually $W$-generic, it follows that there is some finite extension of our condition $p_{n-1}$ to a condition $p_n\in D_n$, which is compatible with $d_0\times\cdots\times d_{n-1}$. Let $d_n$ be the same as $b_n$, except finitely modified to be compatible with $p_n$. In this way, our final real $\oplus_n d_n$ will contain all the conditions $p_n$, and therefore be $W$-generic for $\Add(\omega,\omega)$, yet every $b_n$ will differ only finitely from $d_n$ and hence be an element of $W[d]$. So we have $W[b_0]\cdots[b_n]\of W[d]$, and we have found our upper bound.
\end{proof}

Notice that the real $d$ we construct is not only $W$-generic, but also $W[c_n]$-generic for every $n$.

\bibliographystyle{alpha}
\bibliography{MathBiblio,HamkinsBiblio,WebPosts}

\end{document}

%% file: MathMacrosJDH.tex
\newtheorem{theorem}{Theorem}

\newtheorem{question}[theorem]{Question}

\newtheorem*{questions*}{Questions}
\newtheorem*{mainquestion*}{Main Question} 
\newtheorem*{openquestion*}{Open Question} 

\newcommand{\QED}{\end{proof}}

\def\proclaim[#1]{{\bf #1}}
\def\BF#1.{{\bf #1.}}

\newcommand{\Levy}{L\'{e}vy}

\newcommand{\B}{{\mathbb B}}
\newcommand{\C}{{\mathbb C}}

\renewcommand{\P}{{\mathbb P}}
\newcommand{\Q}{{\mathbb Q}}

\newcommand{\R}{{\mathbb R}}

\newcommand{\of}{\subseteq}

\newcommand{\Add}{\mathop{\rm Add}}

\newcommand{\restrict}{\upharpoonright} 
\newcommand{\satisfies}{\models}

\newcommand{\Union}{\bigcup}

\newcommand{\smalllt}{\mathrel{\mathchoice{\raise2pt\hbox{$\scriptstyle<$}}{\raise1pt\hbox{$\scriptstyle<$}}{\raise0pt\hbox{$\scriptscriptstyle<$}}{\scriptscriptstyle<}}}
\newcommand{\smallleq}{\mathrel{\mathchoice{\raise2pt\hbox{$\scriptstyle\leq$}}{\raise1pt\hbox{$\scriptstyle\leq$}}{\raise1pt\hbox{$\scriptscriptstyle\leq$}}{\scriptscriptstyle\leq}}}

\newcommand{\ltomega}{{{\smalllt}\omega}}

\newcommand{\boolval}[1]{\mathopen{\lbrack\!\lbrack}\,#1\,\mathclose{\rbrack\!\rbrack}}
\def\[#1]{\boolval{#1}}
\newbox\gnBoxA
\newdimen\gnCornerHgt
\setbox\gnBoxA=\hbox{\tiny$\ulcorner$}
\global\gnCornerHgt=\ht\gnBoxA
\newdimen\gnArgHgt
\def\gcode #1{%
\setbox\gnBoxA=\hbox{$#1$}%
\gnArgHgt=\ht\gnBoxA%
\ifnum     \gnArgHgt<\gnCornerHgt \gnArgHgt=0pt%
\else \advance \gnArgHgt by -\gnCornerHgt%
\fi \raise\gnArgHgt\hbox{\tiny$\ulcorner$} \box\gnBoxA %
\raise\gnArgHgt\hbox{\tiny$\urcorner$}}
\newcommand{\UnderTilde}[1]{{\setbox1=\hbox{$#1$}\baselineskip=0pt\vtop{\hbox{$#1$}\hbox to\wd1{\hfil$\sim$\hfil}}}{}}
\newcommand{\Undertilde}[1]{{\setbox1=\hbox{$#1$}\baselineskip=0pt\vtop{\hbox{$#1$}\hbox to\wd1{\hfil$\scriptstyle\sim$\hfil}}}{}}
\newcommand{\undertilde}[1]{{\setbox1=\hbox{$#1$}\baselineskip=0pt\vtop{\hbox{$#1$}\hbox to\wd1{\hfil$\scriptscriptstyle\sim$\hfil}}}{}}
\newcommand{\UnderdTilde}[1]{{\setbox1=\hbox{$#1$}\baselineskip=0pt\vtop{\hbox{$#1$}\hbox to\wd1{\hfil$\approx$\hfil}}}{}}
\newcommand{\Underdtilde}[1]{{\setbox1=\hbox{$#1$}\baselineskip=0pt\vtop{\hbox{$#1$}\hbox to\wd1{\hfil\scriptsize$\approx$\hfil}}}{}}

\def\<#1>{\left\langle#1\right\rangle}

\newcommand{\Ord}{\mathord{{\rm Ord}}}

\newcommand{\ZFC}{{\rm ZFC}}

\newcommand{\DDG}{{\rm DDG}}

\newcommand{\cell}[1]{\boxit{\hbox to 17pt{\strut\hfil$#1$\hfil}}}
\newcommand{\head}[2]{\lower2pt\vbox{\hbox{\strut\footnotesize\it\hskip3pt#2}\boxit{\cell#1}}}
\newcommand{\boxit}[1]{\setbox4=\hbox{\kern2pt#1\kern2pt}\hbox{\vrule\vbox{\hrule\kern2pt\box4\kern2pt\hrule}\vrule}}
\newcommand{\Col}[3]{\hbox{\vbox{\baselineskip=0pt\parskip=0pt\cell#1\cell#2\cell#3}}}
\newcommand{\tapenames}{\raise 5pt\vbox to .7in{\hbox to .8in{\it\hfill input: \strut}\vfill\hbox to
.8in{\it\hfill scratch: \strut}\vfill\hbox to .8in{\it\hfill output: \strut}}}
\newcommand{\Head}[4]{\lower2pt\vbox{\hbox to25pt{\strut\footnotesize\it\hfill#4\hfill}\boxit{\Col#1#2#3}}}
\newcommand{\Dots}{\raise 5pt\vbox to .7in{\hbox{\ $\cdots$\strut}\vfill\hbox{\ $\cdots$\strut}\vfill\hbox{\
$\cdots$\strut}}}
\newcommand{\df}{\it} 